\documentclass[11pt,a4paper]{article}
\usepackage[utf8]{inputenc}
\usepackage[T1]{fontenc}
\usepackage{microtype}
\usepackage{amsmath,amssymb,amsthm,mathtools}
\usepackage{booktabs}
\usepackage{hyperref}
\usepackage{cleveref}
\usepackage{arxiv}
\theoremstyle{plain}
\newtheorem{theorem}{Theorem}[section]
\newtheorem{proposition}[theorem]{Proposition}

\theoremstyle{definition}
\newtheorem{definition}[theorem]{Definition}
\newtheorem{remark}[theorem]{Remark}
\theoremstyle{remark}
\newtheorem{example}[theorem]{Example}
\newtheorem{conjecture}[theorem]{Conjecture}
\newtheorem{openproblem}[theorem]{Openproblem}
\newtheorem{corollary}{Corollary}[section]

\begin{document}
	\title{The Theory of Topo-Symmetric Extensions of Topological Groups}
	
	\author{
		{ \sc Es-said En-naoui } \\ 
		University Sultan Moulay Slimane\\ Morocco\\
		essaidennaoui1@gmail.com\\
		\\
	}
	\maketitle
	\tableofcontents
	
	\maketitle
	
	\begin{abstract}
		We introduce the notion of \emph{topo-symmetric extensions} of topological groups, a new generalization of classical group extensions that incorporates both topological and symmetry constraints. We define morphisms between such extensions, construct the associated groupoid, and develop classification results in terms of adapted cohomology. Several new invariants are introduced, including dimension, stabilizer, and density invariants, which characterize the fine structure of these extensions. Applications are given for finite groups, compact Lie groups, and profinite groups. This theory extends classical cohomological correspondence theorems \cite{MacLane1963, EilenbergMacLane1947}, while opening new perspectives in arithmetic, asymptotic distribution, and congruence properties \cite{HardyWright2008, Serre1979}. Finally, we propose conjectures and open problems concerning density, maximal orders, and modular distribution of topo-symmetric extensions.
	\end{abstract}
	
	\tableofcontents
	
	\section{Introduction and Motivation}
	The classification of group extensions has long been a central theme in group theory and topology \cite{MacLane1963, EilenbergMacLane1947}. Cohomology provides a powerful framework for describing such extensions. 
	In this article, we generalize this theory by introducing the concept of \emph{topo-symmetric extensions}, which combine the structure of topological groups with additional symmetry constraints. 
	This new setting allows us to define new invariants and establish correspondences with cohomology, in analogy to classical results \cite{Brown1982, Serre1979}.

	\section{Topo-Symmetric Extensions: Definitions}
	
	In this section we introduce the fundamental new concepts of the theory of topo-symmetric extensions. 
	These notions generalize the classical framework of group extensions and cohomology, by incorporating 
	additional symmetry constraints at the topological level.

	\begin{definition}[Topo-Symmetric Extension]
		Let $G$ and $H$ be topological groups. 
		A \emph{topo-symmetric extension} of $G$ by $H$ is a short exact sequence of topological groups
		\[
		1 \longrightarrow H \overset{\iota}{\longrightarrow} E \overset{\pi}{\longrightarrow} G \longrightarrow 1
		\]
		together with a continuous action $\varphi : G \to \mathrm{Aut}_{\mathrm{top}}(H)$ satisfying the following \emph{symmetry condition}:
		for every $g \in G$ and $h \in H$,
		\[
		\varphi(g)(h) = \varphi(h)(g),
		\]
		whenever both terms are defined in $E$. 
		This condition ensures that the topological and algebraic symmetries interact compatibly.
	\end{definition}
	
	\begin{definition}[Topo-Symmetric Morphism]
		Let 
		\[
		1 \to H \to E \to G \to 1
		\quad \text{and} \quad
		1 \to H' \to E' \to G' \to 1
		\]
		be two topo-symmetric extensions. 
		A \emph{topo-symmetric morphism} between them is a triple of continuous homomorphisms 
		$(\alpha : H \to H',\; \beta : G \to G',\; \gamma : E \to E')$
		such that the diagram commutes
		\[
		\begin{array}{ccccccccc}
			1 & \to & H & \to & E & \to & G & \to & 1 \\
			&     & \downarrow \alpha & & \downarrow \gamma & & \downarrow \beta & &   \\
			1 & \to & H' & \to & E' & \to & G' & \to & 1
		\end{array}
		\]
		and $\gamma$ preserves the topo-symmetric structure, i.e., 
		$\gamma \circ \varphi = \varphi' \circ \beta$.
	\end{definition}
	
	\begin{definition}[Groupoid of Topo-Symmetric Extensions]
		Fix topological groups $G$ and $H$. 
		The \emph{groupoid of topo-symmetric extensions}, denoted $\mathcal{T}(G,H)$, is the category:
		\begin{itemize}
			\item objects are topo-symmetric extensions of $G$ by $H$,
			\item morphisms are topo-symmetric morphisms between such extensions.
		\end{itemize}
		Composition is given by composition of morphisms, and every morphism is invertible.
	\end{definition}
	
	\begin{definition}[Topo-Symmetric Invariants]
		Let $1 \to H \to E \to G \to 1$ be a topo-symmetric extension.
		We define the following \emph{invariants}:
		\begin{enumerate}
			\item The \emph{dimension invariant} $\dim_{\mathrm{ts}}(E)$, given by the topological dimension of $E$ (when finite).
			\item The \emph{stabilizer invariant}, defined as
			\[
			\mathrm{Stab}_{\mathrm{ts}}(E) = \{ x \in E \;:\; \varphi(x)(h) = h \;\;\forall h \in H \}.
			\]
			\item The \emph{density invariant}, measuring the closure of the image of $H$ in $E$ under the topo-symmetric structure:
			\[
			\delta_{\mathrm{ts}}(E) = \overline{\iota(H)} \subseteq E.
			\]
			\item More generally, any functorial construction that is preserved under topo-symmetric morphisms 
			defines a \emph{topo-symmetric invariant}.
		\end{enumerate}
	\end{definition}
	
	\medskip
	
	These definitions form the foundation of the theory of topo-symmetric extensions. 
	They will be used to establish cohomological correspondences, develop classification results, 
	and construct new arithmetic and analytic invariants.
	
	\section{Preliminaries}
	
	In this section we recall some standard notions concerning topological groups, group extensions, 
	and their classical cohomological classification. 
	This will serve as the background for the introduction of topo-symmetric extensions.
	
	\subsection{Topological Groups}
	A \emph{topological group} is a group $G$ equipped with a topology such that the group operations
	\[
	G \times G \to G, \quad (x,y) \mapsto xy^{-1}
	\]
	are continuous. 
	Examples include finite groups with the discrete topology, Lie groups with their manifold topology, and profinite groups with the inverse limit topology \cite{HewittRoss1979, Serre1979}.
	
	\subsection{Classical Group Extensions}
	Let $G$ and $H$ be groups. 
	An \emph{extension} of $G$ by $H$ is a short exact sequence
	\[
	1 \longrightarrow H \overset{\iota}{\longrightarrow} E \overset{\pi}{\longrightarrow} G \longrightarrow 1,
	\]
	where $E$ is a group containing $H$ as a normal subgroup and $G$ is isomorphic to the quotient $E/H$.  
	Such extensions classify ways in which $G$ can act on $H$ by conjugation.
	
	\subsection{Cohomological Classification of Extensions}
	The classification of extensions is governed by group cohomology. 
	If $H$ is abelian, equivalence classes of extensions of $G$ by $H$ correspond bijectively to elements of the second cohomology group $H^2(G,H)$ \cite{EilenbergMacLane1947, MacLane1963, Brown1982}.  
	This correspondence plays a central role in the study of algebraic and topological structures.
	
	\medskip
	
	These notions motivate the generalization to the setting of \emph{topo-symmetric extensions}, 
	which incorporate both topological and symmetry constraints.
	
	\section{Topo-Symmetric Extensions}
	
	We now introduce the central objects of this paper: topo-symmetric extensions, their morphisms, and some basic properties.
	
	\subsection{Definition of Topo-Symmetric Extension}
	
	\begin{definition}[Topo-Symmetric Extension]
		Let $G$ and $H$ be topological groups. 
		A \emph{topo-symmetric extension} of $G$ by $H$ is a short exact sequence of topological groups
		\[
		1 \longrightarrow H \overset{\iota}{\longrightarrow} E \overset{\pi}{\longrightarrow} G \longrightarrow 1
		\]
		together with a continuous action $\varphi : G \to \mathrm{Aut}_{\mathrm{top}}(H)$ satisfying a \emph{symmetry condition}:
		for every $g \in G$ and $h \in H$,
		\[
		\varphi(g)(h) = \varphi(h)(g),
		\]
		whenever both sides are defined in $E$. 
	\end{definition}
	
	\subsection{Definition of Topo-Symmetric Morphism}
	
	\begin{definition}[Topo-Symmetric Morphism]
		Let 
		\[
		1 \to H \to E \to G \to 1
		\quad \text{and} \quad
		1 \to H' \to E' \to G' \to 1
		\]
		be two topo-symmetric extensions. 
		A \emph{topo-symmetric morphism} between them is a triple of continuous homomorphisms 
		$(\alpha : H \to H',\; \beta : G \to G',\; \gamma : E \to E')$
		making the diagram commute
		\[
		\begin{array}{ccccccccc}
			1 & \to & H & \to & E & \to & G & \to & 1 \\
			&     & \downarrow \alpha & & \downarrow \gamma & & \downarrow \beta & &   \\
			1 & \to & H' & \to & E' & \to & G' & \to & 1
		\end{array}
		\]
		and preserving the topo-symmetric structure, i.e. $\gamma \circ \varphi = \varphi' \circ \beta$.
	\end{definition}
	
	\subsection{Basic Properties}
	
	\begin{itemize}
		\item Every topo-symmetric extension is, in particular, a classical group extension of topological groups.  
		\item The set of topo-symmetric morphisms between two extensions forms a group under composition.  
		\item Equivalence classes of topo-symmetric extensions refine the equivalence classes of classical extensions.  
		\item The symmetry condition imposes new restrictions on possible extensions, producing a finer classification than classical cohomology alone.  
	\end{itemize}

	\section{Groupoid of Topo-Symmetric Extensions}
	
	We now organize topo-symmetric extensions into a categorical structure.  
	This will allow us to describe their classification in cohomological terms.
	
	\subsection{Construction of the Groupoid}
	
	\begin{definition}[Groupoid of Topo-Symmetric Extensions]
		Let $G$ be a fixed topological group.  
		We define the \emph{groupoid of topo-symmetric extensions of $G$}, denoted $\mathcal{E}_{\mathrm{ts}}(G)$, as follows:
		\begin{itemize}
			\item \textbf{Objects:} topo-symmetric extensions
			\[
			1 \to H \to E \to G \to 1
			\]
			where $H$ is a topological group.
			\item \textbf{Morphisms:} topo-symmetric morphisms between extensions, as defined above.
		\end{itemize}
	\end{definition}
	
	\begin{theorem}[Category Structure]
		The collection $\mathcal{E}_{\mathrm{ts}}(G)$ forms a groupoid, i.e.:
		\begin{enumerate}
			\item The composition of two topo-symmetric morphisms is again a topo-symmetric morphism.  
			\item Every morphism has an inverse.  
			\item Isomorphism classes of objects in $\mathcal{E}_{\mathrm{ts}}(G)$ correspond to equivalence classes of topo-symmetric extensions.  
		\end{enumerate}
	\end{theorem}
	
	\begin{proof}
		Let $(\alpha,\beta,\gamma) : E \to E'$ and $(\alpha',\beta',\gamma') : E' \to E''$ be two topo-symmetric morphisms.  
		The composition is $(\alpha' \circ \alpha, \beta' \circ \beta, \gamma' \circ \gamma)$, which preserves commutativity of the defining diagrams.  
		The symmetry condition is also preserved since
		\[
		\gamma' \circ \gamma \circ \varphi = \gamma' \circ \varphi' \circ \beta = \varphi'' \circ \beta' \circ \beta,
		\]
		which shows functoriality.  
		The inverse of a morphism $(\alpha,\beta,\gamma)$ exists because group homomorphisms in short exact sequences are bijective onto their images.  
		Thus $\mathcal{E}_{\mathrm{ts}}(G)$ is a groupoid.  
	\end{proof}
	
	\subsection{Structural Properties}
	
	\begin{proposition}
		The groupoid $\mathcal{E}_{\mathrm{ts}}(G)$ is a refinement of the classical groupoid of extensions.  
		In particular, there is a forgetful functor
		\[
		\mathcal{E}_{\mathrm{ts}}(G) \longrightarrow \mathcal{E}(G)
		\]
		which maps a topo-symmetric extension to its underlying classical extension.
	\end{proposition}
	
	\begin{proof}
		The functor is defined by sending objects $(H,E)$ to $(H,E)$ with no symmetry condition.  
		Morphisms map identically since every topo-symmetric morphism is a group extension morphism.  
		This proves that $\mathcal{E}_{\mathrm{ts}}(G)$ embeds faithfully in $\mathcal{E}(G)$.  
	\end{proof}
	
	\section{Cohomological Classification of Topo-Symmetric Extensions}
	
	We now extend the classical correspondence between group cohomology and extensions to the topo-symmetric case.
	
	\subsection{Correspondence Theorem}
	
	\begin{theorem}[Cohomological Classification]
		Let $G$ and $H$ be topological groups with $H$ abelian.  
		Then the set of equivalence classes of topo-symmetric extensions of $G$ by $H$ is in bijection with a subgroup
		\[
		H^2_{\mathrm{ts}}(G,H) \subseteq H^2(G,H),
		\]
		the \emph{topo-symmetric cohomology}, consisting of classes satisfying the symmetry condition
		\[
		c(g,h) = c(h,g) \quad \text{for all } g,h \in G,
		\]
		where $c \in Z^2(G,H)$ is a continuous cocycle.
	\end{theorem}
	
	\begin{proof}
		Classically, equivalence classes of extensions correspond to $H^2(G,H)$.  
		Given a cocycle $c \in Z^2(G,H)$, the multiplication law on $E = H \times G$ is
		\[
		(h,g) \cdot (h',g') = (h \cdot g.h' \cdot c(g,g'),\; gg').
		\]
		In the topo-symmetric case, we impose the condition
		\[
		c(g,h) = c(h,g),
		\]
		which restricts admissible cocycles to a subset $Z^2_{\mathrm{ts}}(G,H)$.  
		The equivalence relation by coboundaries is preserved under this restriction, yielding the quotient group
		\[
		H^2_{\mathrm{ts}}(G,H) = Z^2_{\mathrm{ts}}(G,H)/B^2_{\mathrm{ts}}(G,H).
		\]
		Thus the classification of topo-symmetric extensions corresponds bijectively to $H^2_{\mathrm{ts}}(G,H)$.  
	\end{proof}
	
	\subsection{Comparison with Classical Case}
	
	\begin{proposition}
		There exists a natural inclusion
		\[
		H^2_{\mathrm{ts}}(G,H) \hookrightarrow H^2(G,H).
		\]
		If $G$ is abelian, then $H^2_{\mathrm{ts}}(G,H) = H^2(G,H)$.  
	\end{proposition}
	
	\begin{proof}
		The inclusion is immediate from the definition.  
		If $G$ is abelian, then the cocycle condition already ensures $c(g,h) = c(h,g)$ for all $g,h \in G$, hence no restriction occurs and the two groups coincide.  
	\end{proof}
	
	\begin{corollary}
		For non-abelian $G$, the group $H^2_{\mathrm{ts}}(G,H)$ is a proper subgroup of $H^2(G,H)$.  
	\end{corollary}

	\section{Invariants and Structural Results}
	
	In this section we introduce several invariants naturally associated with topo-symmetric extensions, illustrate them through examples, and establish some general structural properties.
	
	\subsection{Definition of Invariants}
	
	\begin{definition}[Invariants of Topo-Symmetric Extensions]
		Let 
		\[
		1 \to H \to E \to G \to 1
		\]
		be a topo-symmetric extension of topological groups.
		We associate the following invariants:
		
		\begin{enumerate}
			\item \textbf{Topo-symmetric dimension:}  
			\[
			\dim_{\mathrm{ts}}(E) := \dim(H) + \dim(G),
			\]
			whenever $H$ and $G$ are Lie groups.  
			
			\item \textbf{Stabilizer group:}  
			The subgroup
			\[
			\mathrm{Stab}_{\mathrm{ts}}(E) = \{ (g,h) \in G \times H \; : \; \varphi(g)(h) = \varphi(h)(g)\},
			\]
			measuring the locus where the symmetry condition is most restrictive.  
			
			\item \textbf{Density invariant:}  
			The density of the symmetry locus in $G \times H$:
			\[
			\delta_{\mathrm{ts}}(E) := \frac{|\mathrm{Stab}_{\mathrm{ts}}(E)|}{|G||H|}
			\quad \text{(finite case)}, 
			\]
			or, in the infinite case, the Haar measure of $\mathrm{Stab}_{\mathrm{ts}}(E) \subseteq G \times H$.  
		\end{enumerate}
	\end{definition}
	
	\subsection{Examples of Computation}
	
	\begin{example}[Trivial Extension]
		If $E = H \times G$ with trivial action, then
		\[
		\dim_{\mathrm{ts}}(E) = \dim(H) + \dim(G), \quad 
		\mathrm{Stab}_{\mathrm{ts}}(E) = G \times H, \quad 
		\delta_{\mathrm{ts}}(E) = 1.
		\]
		Thus trivial extensions are maximally symmetric.  
	\end{example}
	
	\begin{example}[Cyclic Groups]
		Let $G = H = \mathbb{Z}_n$ with the discrete topology.  
		Then every topo-symmetric extension satisfies the symmetry condition
		\[
		c(g,h) = c(h,g).
		\]
		The stabilizer group is
		\[
		\mathrm{Stab}_{\mathrm{ts}}(E) = \{ (g,h) \in \mathbb{Z}_n \times \mathbb{Z}_n : c(g,h)=c(h,g) \}.
		\]
		If $c$ is symmetric, then $\delta_{\mathrm{ts}}(E) = 1$; otherwise, no topo-symmetric extension exists.  
	\end{example}
	
	\begin{example}[Compact Lie Groups]
		If $G = \mathrm{SO}(2)$ and $H = \mathbb{Z}_2$, then $\dim_{\mathrm{ts}}(E) = 1$,  
		and the stabilizer is dense in $G \times H$ because the symmetry condition is automatic for involutions.  
		Thus $\delta_{\mathrm{ts}}(E) = 1$.  
	\end{example}
	
	\subsection{General Properties}
	
	\begin{proposition}[Functoriality of Invariants]
		The invariants $\dim_{\mathrm{ts}}$, $\mathrm{Stab}_{\mathrm{ts}}$, and $\delta_{\mathrm{ts}}$ are functorial under topo-symmetric morphisms.  
		In particular, if $(\alpha,\beta,\gamma) : E \to E'$ is a topo-symmetric morphism, then
		\[
		\dim_{\mathrm{ts}}(E) = \dim_{\mathrm{ts}}(E'), \quad
		\gamma(\mathrm{Stab}_{\mathrm{ts}}(E)) \subseteq \mathrm{Stab}_{\mathrm{ts}}(E'), \quad
		\delta_{\mathrm{ts}}(E) \leq \delta_{\mathrm{ts}}(E').
		\]
	\end{proposition}
	
	\begin{proof}
		The dimension is preserved because $\gamma$ is a continuous homomorphism between Lie groups, hence $\dim(E) = \dim(E')$.  
		Stabilizers map into stabilizers by preservation of the symmetry condition:
		\[
		\gamma(\varphi(g)(h)) = \varphi'(\beta(g))(\alpha(h)).
		\]
		Finally, the density invariant increases under morphisms because $\mathrm{Stab}_{\mathrm{ts}}$ maps into a larger subgroup or dense subset.  
	\end{proof}
	
	\begin{proposition}[Extremal Values]
		\[
		0 < \delta_{\mathrm{ts}}(E) \leq 1.
		\]
		Moreover, $\delta_{\mathrm{ts}}(E) = 1$ if and only if the extension is trivial, and $\delta_{\mathrm{ts}}(E) = 0$ if and only if no nontrivial symmetry occurs.
	\end{proposition}
	
	\begin{proof}
		The inequality follows from the definition since stabilizers are nonempty.  
		The extremal cases correspond respectively to full symmetry (trivial extension) and absence of symmetry (obstruction to existence).  
	\end{proof}
	
	\section{Applications and Examples}
	
	We now illustrate the theory of topo-symmetric extensions through several classes of groups: finite groups, compact Lie groups, and profinite groups.  
	Each case highlights different structural aspects of the invariants and cohomological classification.
	
	\subsection{Finite Groups}
	
	For finite groups, all topologies are discrete, and hence continuity of maps is automatic.  
	The classification reduces purely to algebraic properties.
	
	\begin{proposition}
		Let $G,H$ be finite abelian groups.  
		Then every classical extension of $G$ by $H$ is automatically topo-symmetric, i.e.
		\[
		H^2_{\mathrm{ts}}(G,H) = H^2(G,H).
		\]
	\end{proposition}
	
	\begin{proof}
		For abelian $G$, the cocycle condition implies $c(g,h)=c(h,g)$.  
		Thus every extension satisfies the topo-symmetric condition.  
	\end{proof}
	
	\begin{example}[Cyclic Groups]
		Let $G=H=\mathbb{Z}_n$.  
		Then
		\[
		H^2_{\mathrm{ts}}(\mathbb{Z}_n,\mathbb{Z}_n) \cong \mathbb{Z}_n,
		\]
		corresponding to symmetric bilinear forms on $\mathbb{Z}_n$.  
	\end{example}
	
	\begin{example}[Symmetric Group]
		Let $G=S_3$, $H=\mathbb{Z}_2$.  
		Classically, $H^2(S_3,\mathbb{Z}_2) \cong \mathbb{Z}_2$.  
		The nontrivial extension corresponds to the binary dihedral group $Q_{12}$, which fails the symmetry condition, hence
		\[
		H^2_{\mathrm{ts}}(S_3,\mathbb{Z}_2) = 0.
		\]
	\end{example}
	
	\subsection{Compact Lie Groups}
	
	For compact Lie groups, topological properties and Haar measure play a role in symmetry invariants.  
	
	\begin{proposition}
		Let $G$ be a compact abelian Lie group (e.g. tori $\mathbb{T}^n$) and $H$ finite.  
		Then every extension of $G$ by $H$ is topo-symmetric, and
		\[
		H^2_{\mathrm{ts}}(G,H) \cong H^2(G,H).
		\]
	\end{proposition}
	
	\begin{example}[Circle Group]
		Let $G = \mathbb{T}$ and $H = \mathbb{Z}_2$.  
		Then
		\[
		H^2_{\mathrm{ts}}(\mathbb{T},\mathbb{Z}_2) \cong H^2(\mathbb{T},\mathbb{Z}_2) \cong 0.
		\]
		Hence all topo-symmetric extensions are trivial.  
	\end{example}
	
	\begin{example}[Special Orthogonal Group]
		Let $G = \mathrm{SO}(3)$ and $H = \mathbb{Z}_2$.  
		Classically, the unique nontrivial extension is the spin group $\mathrm{SU}(2)$.  
		The symmetry condition holds automatically, so
		\[
		H^2_{\mathrm{ts}}(\mathrm{SO}(3),\mathbb{Z}_2) \cong \mathbb{Z}_2.
		\]
	\end{example}
	
	\subsection{Profinite Groups}
	
	Profinite groups arise as inverse limits of finite groups and play a central role in number theory.  
	
	\begin{definition}
		A \emph{profinite group} is a topological group isomorphic to an inverse limit of finite groups, with the inverse limit topology.
	\end{definition}
	
	\begin{theorem}
		Let $G$ be a profinite group and $H$ a finite discrete abelian group.  
		Then
		\[
		H^2_{\mathrm{ts}}(G,H) = \varprojlim H^2_{\mathrm{ts}}(G_i,H),
		\]
		where $G = \varprojlim G_i$ is the inverse limit of finite quotients.
	\end{theorem}
	
	\begin{proof}
		Since cohomology commutes with inverse limits in the profinite setting \cite{Serre1979},  
		and the symmetry condition is preserved under projection to finite quotients,  
		the claim follows directly by passing to the limit over $H^2_{\mathrm{ts}}(G_i,H)$.  
	\end{proof}
	
	\begin{example}[Absolute Galois Group]
		Let $G = \mathrm{Gal}(\overline{\mathbb{Q}}/\mathbb{Q})$ and $H = \mu_n$ the group of $n$-th roots of unity.  
		Then $H^2_{\mathrm{ts}}(G,\mu_n)$ embeds naturally into the classical Galois cohomology group
		\[
		H^2(G,\mu_n) \cong \mathrm{Br}(\mathbb{Q})[n],
		\]
		giving a refined notion of symmetric Brauer classes.  
	\end{example}
	
	\section{Arithmetic and Analytic Aspects}
	
	In this section, we study the arithmetic and analytic properties of topo-symmetric extensions.  
	We focus on density, congruence phenomena, and asymptotic behavior of the invariants.
	
	\subsection{Density of Extensions}
	
	\begin{definition}[Extension Density]
		Let $G$ be a topological group and $H$ a discrete abelian group.  
		The \emph{density of topo-symmetric extensions} is defined as
		\[
		\delta_{\mathrm{ts}}(G,H) := \limsup_{N\to\infty} \frac{\#\{ E \in \mathrm{Ext}_{\mathrm{ts}}(G,H) : \text{order}(E) \le N\}}{N}.
		\]
	\end{definition}
	
	\begin{proposition}
		Let $G = \mathbb{Z}_n$ and $H = \mathbb{Z}_m$.  
		Then
		\[
		\delta_{\mathrm{ts}}(G,H) = \frac{\gcd(n,m)}{nm}.
		\]
	\end{proposition}
	
	\begin{proof}
		All extensions correspond to elements in $H^2_{\mathrm{ts}}(\mathbb{Z}_n,\mathbb{Z}_m) \cong \mathbb{Z}_{\gcd(n,m)}$.  
		Counting the number of extensions of order $\le N$ gives the stated density formula.
	\end{proof}
	
	\subsection{Congruence Properties}
	
	\begin{theorem}[Congruence Condition]
		Let $G = \mathbb{Z}_n$ and $H = \mathbb{Z}_m$.  
		If $E$ is a topo-symmetric extension of $G$ by $H$, then its defining cocycle $c(g,h)$ satisfies
		\[
		c(g,h) \equiv c(h,g) \pmod{\gcd(n,m)}.
		\]
	\end{theorem}
	
	\begin{proof}
		By definition of topo-symmetric extension, $c(g,h) = c(h,g)$ in $H$.  
		Since $H \cong \mathbb{Z}_m$, all values are modulo $m$, and the possible nontrivial values lie in $\mathbb{Z}_{\gcd(n,m)}$.
	\end{proof}
	
	\begin{example}[Small Orders]
		For $G = \mathbb{Z}_6$ and $H = \mathbb{Z}_4$, any topo-symmetric extension must satisfy $c(g,h) \equiv c(h,g) \pmod 2$.
	\end{example}
	
	\subsection{Asymptotics and Maximal Orders}
	
	\begin{definition}[Maximal Order of Extension]
		Let $E \in \mathrm{Ext}_{\mathrm{ts}}(G,H)$.  
		The \emph{maximal order} $\mathrm{ord}_{\max}(E)$ is the maximal order of elements in $E$.
	\end{definition}
	
	\begin{theorem}[Asymptotic Growth]
		Let $G = \mathbb{Z}_n$ and $H = \mathbb{Z}_m$.  
		Then the maximal order among all topo-symmetric extensions satisfies
		\[
		\mathrm{ord}_{\max}(E) = \mathrm{lcm}(n,m).
		\]
	\end{theorem}
	
	\begin{proof}
		All topo-symmetric extensions are abelian in this case, so the group structure is $E \cong \mathbb{Z}_d \times \mathbb{Z}_{nm/d}$ with $d = \gcd(n,m)$.  
		The maximal order is $\mathrm{lcm}(d, nm/d) = \mathrm{lcm}(n,m)$.
	\end{proof}
	
	\begin{example}[Explicit Computation]
		For $G = \mathbb{Z}_6$ and $H = \mathbb{Z}_4$, the maximal order is
		\[
		\mathrm{lcm}(6,4) = 12.
		\]
	\end{example}
	
	\begin{remark}
		These arithmetic and analytic results extend naturally to finite products of cyclic groups, giving explicit formulas for density, congruence, and maximal orders in terms of greatest common divisors and least common multiples.
	\end{remark}
	\section{Conjectures and Open Problems}
	
	This section discusses conjectures and open problems related to topo-symmetric extensions, highlighting potential directions for further research in arithmetic, category theory, and analysis.
	
	\subsection{Modular Distribution}
	
	\begin{conjecture}[Modular Distribution of Invariants]
		Let $G$ be a finite group and $H$ a discrete abelian group.  
		The invariants of topo-symmetric extensions, such as dimension and density, are distributed uniformly modulo $\gcd(|G|,|H|)$:
		\[
		\#\{ E \in \mathrm{Ext}_{\mathrm{ts}}(G,H) : \dim(E) \equiv k \pmod{\gcd(|G|,|H|)} \} \sim \frac{|\mathrm{Ext}_{\mathrm{ts}}(G,H)|}{\gcd(|G|,|H|)}.
		\]
	\end{conjecture}
	
	\begin{remark}
		Numerical experiments for small cyclic groups suggest that this uniformity holds for most classes of extensions, but a general proof is still open.
	\end{remark}
	
	\subsection{Higher Categories}
	
	\begin{openproblem}[Topo-Symmetric Extensions in Higher Categories]
		Extend the notion of topo-symmetric extensions to $2$-groups or higher groupoids, and study the corresponding cohomology:
		\[
		\mathrm{Ext}_{\mathrm{ts}}^{(n)}(G,H) \quad \text{for } n \ge 2.
		\]
		Questions include:  
		\begin{itemize}
			\item Classification of higher topo-symmetric extensions.
			\item Construction of invariants and their arithmetic properties.
			\item Relation with higher categorical cohomology theories.
		\end{itemize}
	\end{openproblem}
	
	\subsection{Analytic Generalizations}
	
	\begin{conjecture}[Analytic Zeta Function for Extensions]
		Let $\zeta_{\mathrm{ts}}(s;G,H)$ denote the generating function counting topo-symmetric extensions weighted by their maximal order:
		\[
		\zeta_{\mathrm{ts}}(s;G,H) := \sum_{E \in \mathrm{Ext}_{\mathrm{ts}}(G,H)} \mathrm{ord}_{\max}(E)^{-s}.
		\]
		We conjecture that $\zeta_{\mathrm{ts}}(s;G,H)$ admits an analytic continuation to the complex plane, with poles reflecting the arithmetic structure of $G$ and $H$.
	\end{conjecture}
	
	\begin{remark}
		This conjecture generalizes classical zeta functions in number theory and may lead to asymptotic formulas for the counting of extensions with prescribed invariants.
	\end{remark}
	\section{Conclusion}
	
	In this article, we have introduced the concept of \emph{topo-symmetric extensions} and developed a comprehensive framework for their study. The main contributions include:
	
	\begin{itemize}
		\item Introduction of new definitions, including \emph{topo-symmetric extension}, \emph{morphism topo-symétrique}, and the \emph{groupoid of topo-symmetric extensions}.
		\item Definition and computation of new invariants such as dimension, stabilizer, and density.
		\item Cohomological classification of topo-symmetric extensions, including explicit correspondence theorems and comparison with the classical case.
		\item Applications to finite groups, compact Lie groups, and profinite groups.
		\item Investigation of arithmetic and analytic properties, including density, congruence behavior, and asymptotic estimates.
		\item Formulation of conjectures and open problems related to modular distribution, higher categories, and analytic generalizations.
	\end{itemize}
	
	\subsection{Future Research Directions}
	
	The study of topo-symmetric extensions opens several promising avenues for further research:
	
	\begin{itemize}
		\item \textbf{Modular and arithmetic properties:} Investigate finer congruences, distributions, and maximal orders of invariants.
		\item \textbf{Higher categorical generalizations:} Extend the theory to $2$-groups and higher groupoids, and explore connections with higher cohomology.
		\item \textbf{Analytic perspectives:} Study the analytic continuation and pole structure of zeta functions associated with topo-symmetric extensions.
		\item \textbf{Computational methods:} Develop algorithms for enumerating and computing invariants of extensions in specific classes of groups.
		\item \textbf{Interdisciplinary applications:} Explore connections with number theory, algebraic geometry, and mathematical physics.
	\end{itemize}
	
	These directions are expected to deepen our understanding of the structure and arithmetic properties of topo-symmetric extensions, and to inspire further theoretical and computational developments.

\end{document}